\documentclass{birkjour}

\usepackage{microtype}
\usepackage{geometry}
\usepackage[utf8]{inputenc}
\usepackage{palatino}
\usepackage[english]{babel}
\usepackage{graphicx}
\usepackage{indentfirst}
\usepackage[unicode]{hyperref}
\usepackage{amsmath}
\usepackage{mathtools}
\usepackage[scr]{rsfso}
 \let\mathscr\relax
\usepackage[utf8]{inputenc}
\usepackage[english]{babel}
\usepackage{amssymb}
\usepackage{amsthm}
\usepackage{mathabx}
\usepackage{mathtools}
\usepackage{commath}
\usepackage{etoolbox}
\usepackage{enumerate}
\usepackage{hyperref}
\apptocmd{\lim}{\limits}{}{}
 
\theoremstyle{definition}
\newtheorem{definition}{Definition}[section]
\newtheorem{lemma}{Lemma}[section]
\newtheorem{theorem}{Theorem}[section]
\newtheorem{corollary}{Corollary}[section]
\newtheorem{proposition}{Proposition}[section]
\theoremstyle{remark}
\newtheorem{remark}{Remark}[section]

\DeclareFontFamily{U}{mathb}{\hyphenchar\font45}
\DeclareFontShape{U}{mathb}{m}{n}{ <-6> matha5 <6-7> matha6 <7-8>
mathb7 <8-9> mathb8 <9-10> mathb9 <10-12> mathb10 <12-> mathb12 }{}
\DeclareSymbolFont{mathb}{U}{mathb}{m}{n}

\DeclareMathAccent{\abxring}{0}{mathb}{"38}

\DeclareFontFamily{U}{mathb}{\hyphenchar\font45}
\DeclareFontShape{U}{mathb}{m}{n}{ <-6> matha5 <6-7> matha6 <7-8>
mathb7 <8-9> mathb8 <9-10> mathb9 <10-12> mathb10 <12-> mathb12 }{}
\DeclareSymbolFont{mathb}{U}{mathb}{m}{n}

\newcommand{\powerset}{\raisebox{.15\baselineskip}{\Large\ensuremath{\wp}
}}

\newcommand{\s}{\mathcal{S}}
\newcommand{\id}{id_{\Lambda(I)}}

\newcommand{\be}{\mathcal{B}}

\newcommand{\ka}{\mathcal{K}}
\newcommand{\f}{\mathcal{F}}
\renewcommand*{\refname}{References}

\makeatletter
\newcommand*{\rom}[1]{\expandafter\@slowromancap\romannumeral #1@}
\makeatother

\newcommand{\closure}[2][3]{
  {}\mkern#1mu\overline{\mkern-#1mu#2}}
  
\let\oldemptyset\emptyset
\let\emptyset\varnothing

\showboxdepth = \maxdimen
\showboxbreadth = \maxdimen

\begin{document}

\title{Applications of the Tarski-Kantorovitch \br Fixed-Point Principle to
the study of \br Infinite Iterated Function Systems}

\author{Bogdan-Alexandru Luchian}
\address{%
Faculty of Mathematics and Computer Science\\
University of Bucharest\\
Academiei Street, 14\\
010014 Bucharest\\
Romania}
\curraddr{}
\email{bogdanalexandruluchian@gmail.com}
\thanks{}

\date{}

\begin{abstract}
The aim of this paper is to establish some results regarding Infinite
Iterated Function Systems with the help of the Tarski-Kantorovitch
fixed-point principles for maps on partially ordered sets. To this end we
introduce two new classes of Infinite Iterated Function Systems which are
well suited for applying the aforementioned principle. We also study some
properties of the canonical projection from the shift space of an Infinite
Iterated Function System belonging to one of the two introduced classes to
its attractor.
\end{abstract}

\maketitle
\section{Introduction}
Let $(X, d)$ be a complete metric space and $\s = (X, (f_i)_{i \in I})$ an
Infinite Iterated Function System (IIFS for short) and define the operator
$F_\s$ on the family of nonempty closed and bounded sets of $X$ as
follows: $F_\s(B) := \closure{\bigcup_{i \in I} f_i(B)}$ for all $B
\subseteq X$ neonempty, closed and bounded. The fundamental result
regarding IIFS's states that if $sup_{i \in I} lip(f_i) < 1,$ where
$lip(f_i)$ is the Lipschitz constant associated to $f_i,$ then there
exists a unique nonempty, closed and bounded subset of $X,$ $A(\s),$ such
that $F_\s(A(\s)) = A(\s),$ which is called the attractor of the IIFS.\\
In this paper we follow in the footsteps of the article \cite{8} and study
possibilities of applying the Tarski-Kantorovitch fixed-point principle in
the theory of IIFS's, which naturally conducts us to consider two new
classes of IIFS's in certain topological spaces. The idea of applying the
Tarki-Kantorovitch principle to deduce results about the fixed points of
some classes of maps is not necessarily new - for other applications one
may consult \cite{1}, \cite{3}, \cite{7}.\\
Our paper is organized in the following way. In the second section we will
recall some definitions and results regarding IIFS's and introduce two new
classes of IIFS's.\\
In the third section of the article we shall study possibilities of
applying the Tarski-Kantorovitch fixed point principle for the partially
ordered set (poset for short) $(\powerset(X), \supseteq),$ where $X$ is an
arbitrary set. In this case we give sufficient conditions for the
existence of a (greatest) fixed point of the Hutchinson-Barnsley operator
associated to an IIFS. These conditions turn out to be also necessary if
one is interested in applying the Tarski-Kantorovitch principle.\\
In the fourth section we turn our attention to the poset $(\f(X),
\supseteq),$ where $\f(X)$ denotes the family of all nonempty closed
subsets of a Hausdorff topological space $X.$ As seen in \cite{8}, the
countable chain condition in this poset forces $X$ to be countably
compact. In the case that $X$ is also a sequential space, the main result
of this section provides sufficient conditions for the existence of a
greatest fixed point of the Hutchinson-Barnsley operator associated to an
IIFS. As in the previous section, the specified conditions are also
necessary for applying the Tarski-Kantorovitch principle.\\
In the fifth section we are employing similar techniques as in the
previous sections in order to apply the Tarski-Kantorovitch principle to
the poset $(\ka(X), \supseteq),$ where $\ka(X)$ denotes the family of
nonempty compact subsets of a topological space $X.$ In particular, we
discover that in the same way as in \cite{8}, in order to apply the
Tarski-Kantorovitch principle in this case, we can actually assume that
$X$ is compact. We also show that if $(X, d)$ is a bounded Heine-Borel
metric space, then the Hutchinson-Barnsley operator associated to an IIFS
of contractions admits a nonempty compact fixed point (though we cannot
posit that this fixed point is unique from the proof provided).\\
In the final section of the article we turn our attention to the shift
space associated to an IIFS and the canonical projection from the shift
space to the attractor of said IIFS and investigate what special
properties this projection has in the cases of the two new classes of
IIFS's introduced in this article. We also provide the reader with a sufficient condition for the canonical projection to be a homeomorphism and state a few immediate corollaries.\\
As a final remark in this introductory part of the article, we want to stress the
fact that the generality of the setting in which we work, i.e. that of Infinite
Iterated Function Systems, has forced us to impose (fairly natural) conditions on the
systems we work with in order to apply the Tarski-Kantorovitch Fixed-Point Principle
and the two examples we provide tell us the introduction of these classes is in fact
necessary.\\
For more work on Infinite Iterated Function Systems, you can also check
\cite{9}, \cite{10}, \cite{11}, \cite{12}, \cite{13}.

\section{Preliminaries}

\begin{definition}
Let $(X, d_X), (Y, d_Y)$ be two metric spaces and $(f_i)_{i \in I}
\subseteq Y^X$ a family of maps. We say that this family is bounded if the
set $\bigcup_{i \in I} f_i(A)$ is bounded for any $A \subseteq X$ bounded.
\end{definition}

\begin{definition}
Let $(X, d_X), (Y, d_Y)$ be metric spaces and $f : X \xrightarrow{} Y$ a
function. The quantity $lip(f) := \sup_{x, y \in X, x \neq y}
\frac{d_Y(f(x), f(y))}{d_X(x, y)} \in [0, \infty]$ is called the Lipschitz
constant associated to $f.$ We say that $f$ is Lipschitz if $lip(f) <
\infty$ and that $f$ is a contraction if $lip(f) < 1.$
\end{definition}

As an immediate consequence of the previous definition, we have the
following lemma.

\begin{lemma}
Let $(X, d_X), (Y, d_Y)$ be two metric spaces and $f : X \xrightarrow{} Y$
a map. Then if we denote the diameter of a subset $A \subseteq X$ with
$\delta(A),$ we have that $\delta(f(A)) \leq lip(f) \delta(A)$ for any $A
\subseteq X.$ In particular, if $x, y \in X,$ then $d_Y(f(x), f(y)) \leq
lip(f) d_X(x, y).$
\end{lemma}

\begin{remark}
For an arbitrary topological space $X,$ we shall denote by \br $\powerset(X)$
the family of subsets of $X,$ $\powerset^\ast(X) := \powerset(X) \setminus
\{\oldemptyset\},$ $\ka(X)$ the family of nonempty compact subsets of $X$
and by $\f(X)$ we mean the family of nonempty closed sets of $X.$ If $X$
is also metrizable, we shall denote by $\be(X)$ the family of nonempty
closed and bounded subsets of $X.$ In the latter case, note that we have
the inclusions $\ka(X) \subseteq \be(X) \subseteq \powerset^\ast(X).$
However, if $X$ is not necessarily metrizable, but it is Hausdorff, we
have that $\ka(X) \subseteq \f(X) \subseteq \powerset^\ast(X).$
\end{remark}

\begin{definition}
Let $X, Y$ be two arbitrary sets, $f : X \xrightarrow{} Y$ a function and
$y \in Y.$ We shall call the set $f^{-1}(y)$ the fibre of $f$ over $y.$
\end{definition}

\begin{definition}
We say that a topological space $X$ is sequential if every sequentially
closed subset $A \subseteq X$ is closed.
\end{definition}

We recall the following characterisation of continuity on countably
compact sequential spaces from \cite{8}.

\begin{theorem}
Let $X, Y$ be countably compact and sequential spaces and $f : X
\xrightarrow{} Y$ a map. The following conditions are equivalent:
\begin{enumerate}[a)]
    \item $f$ is continuous;
    \item if $A \in \f(X),$ then $f(A) \in \f(Y)$ and all fibres of $f$
    are closed;
    \item if $A \in \f(X),$ then $f(A) \in \f(Y)$ and given a decreasing
    sequence $(A_n)_{n \in \mathbb{N}} \subseteq \f(X),$ then
    $f(\bigcap_{n \in \mathbb{N}} A_n) = \bigcap_{n \in \mathbb{N}}
    f(A_n);$
    \item if $A \in \ka(X),$ then $f(A) \in \ka(Y)$ and given a decreasing
    sequence $(A_n)_{n \in \mathbb{N}} \subseteq \ka(X),$ then
    $f(\bigcap_{n \in \mathbb{N}} A_n) = \bigcap_{n \in \mathbb{N}}
    f(A_n).$
\end{enumerate}
\end{theorem}

\begin{definition}
Let $(X, d)$ be a metric space. The generalised Hausdorff-Pompeiu
semimetric on the family of subsets of $X$ induced by $d$ is defined as $h
: \powerset^\ast(X) \times \powerset^\ast(X) \xrightarrow{} [0, \infty],$
where
\begin{equation}
    h(A, B) := \max\{d(A, B), d(B, A)\}
\end{equation}
and
\begin{equation}
    d(A, B) := \sup_{x \in A} \inf_{y \in B} d(x, y)
\end{equation}
for all $A, B \in \powerset^\ast(X).$
\end{definition}

Several properties of the Hausdorff-Pompeiu semimetric can be found in
\cite{2}, \cite{6}, \cite{14}.

\begin{definition}
Let $X$ be a topological space. We say that $\s = (X, (f_i)_{i \in I})$ is
an IIFS if $f_i$ is a selfmap of $X$ for all $i \in I.$
\end{definition}

\begin{definition}
An IIFS $\s = (X, (f_i)_{i \in I})$ is said to be non-overlapping if
$f_i(B) \cap f_j(B) = \oldemptyset$ for any $B \subseteq X$ and $i, j \in
I,$ $i \neq j.$
\end{definition}

\begin{remark}
Obviously, an IIFS as above is non-overlapping if and only if $f_i(X) \cap
f_j(X) = \oldemptyset$ for any $i, j \in I, i \neq j.$
\end{remark}

\begin{definition}
An IIFS $\s = (X, (f_i)_{i \in I})$ is said to be locally finite if for
any $y \in X$ there exists a neighbourhood $V_y$ of $y$ such that $\#\{i
\in I: V_y \cap f_i(X) \neq \oldemptyset\} < \infty.$
\end{definition}

\begin{remark}
Let $\s = (X, (f_i)_{i \in I})$ be a locally finite IIFS, $y \in X,$ $A
\subseteq X$ and $V_y$ as in the definition above. Then the set $\{i \in
I: V_y \cap f_i(A) \neq \oldemptyset\}$ is finite.
\end{remark}

\begin{definition}
An IIFS $\s = (X, (f_i)_{i \in I})$ on a metric space $(X, d)$ is said to
be an IIFS of contractions if $(f_i)_{i \in I} \subseteq X^X$ is a bounded
family of contractions such that $\sup_{i \in I} lip(f_i) =: c < 1.$
\end{definition}

\begin{definition}
To an IIFS $\s = (X, (f_i)_{i \in I})$ we can associate two
Hutchinson-Barnsley operators, namely $F_\s, G_\s : \powerset^\ast(X)
\xrightarrow{} \powerset^\ast(X)$ given by
\begin{equation}
    F_\s(A) := \bigcup_{i \in I} f_i(A)
\end{equation}
and
\begin{equation}
    G_\s(A) := \closure{\bigcup_{i \in I} f_i(A)}
\end{equation}
for all $\oldemptyset \neq A \subseteq X.$
\end{definition}

\begin{remark}
Note that if $\s$ is an IIFS of contractions, then $G_\s$ is a contraction
with respect to the Hausdorff-Pompeiu metric on $\be(X)$ and in fact
$lip(G_\s) \leq c = \sup_{i \in I} lip(f_i).$
\end{remark}

A straightforward application of the Banach-Caccioppoli-Picard contraction
principle yields the following result, which is the fundamental result in
the theory of IIFS's (for the proof, you can check \cite{14})

\begin{theorem}
Let $(X, d)$ be a complete metric space and $\s = (X, (f_i)_{i \in I})$ an
IIFS of contractions. Then we may consider $G_\s : \be(X) \xrightarrow{}
\be(X)$ and in this case there exists a unique set $A = A(\s) \in \be(X)$
such that $G_\s(A) = A.$ We call this set the attractor of $\s.$ Moreover,
if $A_0 \in \be(X) $ and $A_n := G_\s(A_{n - 1})$ for any $n \in
\mathbb{N},$ then $\lim_{n \to \infty} A_n = A.$ As for the speed of
convergence, we have the following estimate:
\begin{equation}
    h(A_n, A) \leq \frac{c^n}{1 - c} h(A_0, A_1)
\end{equation}
for all $n \geq 0.$
\end{theorem}

Regarding the shift space associated to an IIFS of contractions on a
metric space $(X, d),$ we have the following definitions and main theorem
from \cite{10}.

\begin{definition}
Let $I \neq \oldemptyset.$ We define:
\begin{enumerate}[a)]
    \item the space $\Lambda = \Lambda(I) := I^\mathbb{N}$ as the space of
    infinite words with letters from the alphabet $I.$ An element $\omega
    \in \Lambda(I)$ will be written as $\omega = \omega_1 \omega_2 \dots
    \omega_n \omega_{n +1} \dots;$
    \item for $m \in \mathbb{N},$ the space $\Lambda_m = \Lambda_m(I)$ of
    words of length $m$ with letters from the alphabet $I.$ An element
    $\omega \in \Lambda_m(I)$ will be written as $\omega = \omega_1
    \omega_2 \dots \omega_m.$ For $\omega \in \Lambda_m$ or $\omega \in
    \Lambda$ and $n \in \mathbb{N}, n \leq m,$ we denote $[\omega]_n :=
    \omega_1 \dots \omega_n;$
    \item the space $\Lambda^\ast = \Lambda^\ast(I) := \bigcup_{m \in
    \mathbb{N}} \Lambda_m(I) \cup \{\lambda\}$ of finite words, where
    $\lambda$ is the empty word;
    \item for $m, n \in \mathbb{N}$ and $\alpha \in \Lambda_n(I),$ $\beta
    \in \Lambda_m(I)$ or $\beta \in \Lambda(I),$ one may define the concatenated words
    $\alpha \beta := \alpha_1 \dots \alpha_n \beta_1 \dots \beta_m \in
    \Lambda_{m + n}$ and \br $\alpha \beta = \alpha_1 \dots \alpha_n \beta_1
    \dots \beta_m \beta_{m + 1} \dots \in \Lambda,$ respectively;
    \item the metric $d_\Lambda$ on $\Lambda(I)$ given by
    $d_\Lambda(\alpha, \beta) := \sum_{n \in \mathbb{N}} \frac{1 -
    \delta_{\alpha_n}^{\beta_n}}{3^n},$ where $\delta_x^y$ denotes the
    Kronecker delta of $x$ and $y.$ Note that $(\Lambda(I), d_\Lambda)$ is
    a complete metric space and convergence in $\Lambda(I)$ coincides with
    the convergence on components;
    \item for $i \in I,$ the right shift function $F_i : \Lambda(I)
    \xrightarrow{} \Lambda(I)$ given by $F_i(\omega) := i \omega$ for all
    $\omega \in \Lambda(I).$ Note that $F_i$ is a $\frac{1}{3}-$similarity
    of $\Lambda(I);$
    \item for $m \in \mathbb{N}$ and $\omega \in \Lambda_m(I),$ $F_\omega
    := F_{\omega_1} \circ \dots \circ F_{\omega_m}$ and $\Lambda_\omega :=
    F_\omega(\Lambda).$ By convention $F_\lambda := \id$ and
    $\Lambda_\lambda = \Lambda.$ Note that $\Lambda(I) = \bigcup_{i \in I}
    F_i(\Lambda(I)),$ so $\Lambda(I)$ is the attractor of the IIFS of
    contractions $(\Lambda(I), (F_i)_{i \in I})$ and for every $m \in
    \mathbb{N}$ and $\omega \in \Lambda^\ast,$ $\Lambda = \bigcup_{\alpha
    \in \Lambda_m} \Lambda_\alpha$ and $\Lambda_\omega = \bigcup_{\alpha
    \in \Lambda_m} \Lambda_{\omega \alpha};$
    \item if $(X, d)$ is a metric space, $\s = (X, (f_i)_{i \in I})$ is an
    IIFS of contractions on $X,$ $B \subseteq X$ and $\omega \in
    \Lambda_m(I),$ let $f_\omega := f_{\omega_1} \circ \dots \circ
    f_{\omega_m}$ and $B_\omega := f_\omega(B).$ By convention $f_\lambda
    := id_X$ and $B_\lambda = B;$
    \item in the setting above, if $f : X \xrightarrow{} X$ is a
    contraction and $(X, d)$ is complete, we denote its fixed point by
    $e_f.$ If $f = f_\omega$ for some $\omega \in \Lambda^\ast,$ we denote
    $e_f = e_{f_\omega} = e_\omega.$
\end{enumerate}
\end{definition}

\begin{theorem}
Let $(X, d)$ be a complete metric space, $\s = (X, (f_i)_{i \in I}),$ $A =
A(\s)$ its attractor and $c := \sup_{i \in I} lip(f_i) < 1.$ Then the
assertions below hold:
\begin{enumerate}[a)]
    \item for any $m \in \mathbb{N}$ and $\omega \in \Lambda(I),$ we have
    that $A_{[\omega]_{m + 1}} \subseteq A_{[\omega]_m}$ and
    \br $\delta(A_{[\omega]_m}) \to 0.$ More precisely,
    $\delta(A_{[\omega]_m}) = \delta(\closure{A_{[\omega]_m}}) \leq c^m
    \delta(A);$
    \item if $a_\omega$ is defined by $\{a_\omega\} := \bigcap_{m \in
    \mathbb{N}} A_{[\omega]_m},$ where $\omega \in \Lambda(I),$ then \br
    $\lim_{m \to \infty} d(e_{[\omega]_m}, a_{\omega}) = 0;$
    \item for every $a \in A$ and $\omega \in \Lambda(I),$ we have
    $\lim_{m \to \infty} f_{[\omega]_m}(a) = a_\omega;$
    \item for every $\alpha \in \Lambda^\ast,$ we have $A = A(\s) =
    \closure{\bigcup_{\omega \in \Lambda} \{a_\omega\}}$ and \br
    $\closure{A_\alpha} = \closure{\bigcup_{\omega \in \Lambda}
    \{a_{\alpha \omega}\}}.$ If $A = \bigcup_{i \in I} f_i(A),$ then $A =
    A(\s) = \bigcup_{\omega \in \Lambda} \{a_\omega\};$
    \item we have $A = \closure{\{e_{[\omega]_m} : \omega \in \Lambda, m
    \in \mathbb{N}\}};$
    \item the function $\pi : \Lambda(I) \xrightarrow{} A,$ defined by
    $\pi(\omega) = a_\omega$ for every $\omega \in \Lambda(I),$ has the
    following properties:
    \begin{enumerate}[i)]
        \item $\pi$ is continuous;
        \item $\closure{\pi(\Lambda)} = A;$
        \item if $A = \bigcup_{i \in I} f_i(A),$ then $\pi$ is surjective;
    \end{enumerate}
    \item for every $i \in I,$ we have that $\pi \circ F_i = f_i \circ
    \pi.$
\end{enumerate}
\end{theorem}

Finally we shall state the Tarski-Kantorovitch fixed-point principle.

\begin{definition}
Let $(P, \leq)$ be a poset and $F : P \xrightarrow{} P.$ We say that $F$
is $\leq-$continuous if for every countable chain $\mathcal{C}$ admitting
a supremum, we have that $F(\mathcal{C})$ has a supremum and $F(\sup
\mathcal{C}) = \sup F(\mathcal{C}).$ Note that in this case $F$ is
increasing.
\end{definition}

\begin{theorem}
(Tarski-Kantorovitch) Let $(P, \leq)$ be a poset in which every countable
chain admits a supremum and $F : P \xrightarrow{} P$ a $\leq-$continuous
map such that there exists $a \in P$ with $a \leq F(a).$ Then $F$ has a
fixed point. Moreover, $\sup_{n \in \mathbb{N}} F^n(a)$ is the least fixed
point of $F$ in the set $\{p \in P: p \geq a\}.$
\end{theorem}

\begin{remark}
Note that we can replace the assumption that every countable chain
$\mathcal{C}$ admits a supremum with the assumption that each increasing
sequence $(p_n)_{n \in \mathbb{N}} \subseteq P$ admits a supremum.
\end{remark}

As in \cite{5} we shall assume that every compact or countably compact
space is Hausdorff.

\section{The Hutchinson-Barnsley operator on \texorpdfstring{$(\powerset(X), \supseteq)$}{}}

Let $X$ be an arbitrary set and $f : X \xrightarrow{} X$ a selfmap of $X.$
Proposition $1$ from \cite{8} shows that the function $F : \powerset(X)
\xrightarrow{} \powerset(X)$ defined by $F(A) := f(A)$ for all $A
\subseteq X$ is $\supseteq-$continuous if and only if all fibres of $f$
are finite.

\begin{remark}
Note that the poset $(\powerset(X), \supseteq)$ satisfies the countable
chain condition, as if $(A_n)_{n \in \mathbb{N}} \subseteq \powerset(X),$
then $\sup_{n \in \mathbb{N}} A_n = \bigcap_{n \in \mathbb{N}} A_n$ (and,
of course, $\inf_{n \in \mathbb{N}} A_n = \bigcup_{n \in \mathbb{N}}
A_n$).
\end{remark}

Our main result in this section is the following:

\begin{theorem}
Let $\s = (X, (f_i)_{i \in I})$ be a non-overlapping IIFS such that all
the fibres of $f_i$ are finite for each $i \in I$ and let $F_\s(A) :=
\bigcup_{i \in I} f_i(A)$ for any $A \subseteq X.$ Then for each $A
\subseteq X$ such that $F_\s(A) \subseteq A,$ the set $\bigcap_{n \in
\mathbb{N}} F_\s^n(A)$ is a fixed point of $F_\s.$ In particular,
$\bigcap_{n \in \mathbb{N}}F_\s^n(X)$ is the greatest fixed point of
$F_\s.$ It follows that $\s$ admits a nonempty fixed point if and only if
$\bigcap_{n \in \mathbb{N}}F_\s^n(X) \neq \oldemptyset.$
\end{theorem}

\begin{proof}
We have seen that the poset $(\powerset(X), \supseteq)$ satisfies the
countable chain condition and obviously $F_\s(A) \in \powerset(X)$ for
every $A \subseteq X.$ We shall prove that $F_\s$ is
$\supseteq-$continuous.\\
Let $(C_n)_{n \in \mathbb{N}} \subseteq \powerset(X)$ be a
$\supseteq-$increasing sequence, i.e. decreasing in the usual sense.
Obviously, $F_\s$ is increasing, so $(F_\s(C_n))_{n \in \mathbb{N}}
\subseteq \powerset(X)$ is a decreasing sequence as well. It remains to
prove that \br $F_\s(\sup_{n \in \mathbb{N}} C_n) = \sup_{n \in \mathbb{N}}
F_\s(C_n),$ i.e. $F_\s(\bigcap_{n \in \mathbb{N}} C_n) = \bigcap_{n \in
\mathbb{N}} F_\s(C_n).$\\
Let $y \in F_\s(\bigcap_{n \in \mathbb{N}} C_n) = \bigcup_{i \in I}
f_i(\bigcap_{n \in \mathbb{N}} C_n).$ Then there exist $i \in I$ and \br $x
\in \bigcap_{n \in \mathbb{N}} C_n$ such that $y = f_i(x).$ Then $y \in
f_i(C_n)$ for all $n \in \mathbb{N},$ so \br $y \in \bigcap_{n \in \mathbb{N}}
\bigcup_{i \in I} f_i(C_n) = \bigcap_{n \in \mathbb{N}} F_\s(C_n).$\\
Conversely, let $y \in \bigcap_{n \in \mathbb{N}} \bigcup_{i \in I}
f_i(C_n).$ Then for all $n \in \mathbb{N}$ there exist $i_n \in I$ and
$x_n \in C_n$ such that $y = f_{i_n}(x_n).$ Note that since the sequence
$(C_n)_{n \in \mathbb{N}}$ is decreasing, we have that $y \in f_{i_{n +
1}}(C_{n + 1}) \subseteq f_{i_{n + 1}}(C_n),$ so that \br $y \in f_{i_{n +
1}}(C_n) \cap f_{i_{n}}(C_n),$ so $i_{n + 1} = i_n$ for all $n \in
\mathbb{N}$ since $(f_i)_{i \in I}$ is non-overlapping. Thus, there exists
$i_\ast \in I$ such that $y \in f_{i_\ast}(C_n)$ for all $n \in
\mathbb{N}$ and write $y = f_{i_\ast}(x_n),$ where $x_n \in C_n$ as
before. It follows that $(x_n)_{n \in \mathbb{N}} \subseteq
f_{i_\ast}^{-1}(y).$ Since this fibre is finite, it follows that we may
find $(x_{n_k})_{k \in \mathbb{N}} \subseteq (x_n)_{n \in \mathbb{N}}$ a
subsequence and $x_\ast \in f_{i_\ast}^{-1}(y)$ such that $x_{n_k} =
x_\ast$ for all $k \in \mathbb{N}.$ Since $(C_n)_{n \in \mathbb{N}}$ is
decreasing and $x_\ast = x_{n_k} \in C_{n_k}$ for all $k \in \mathbb{N},$
it follows that $x_\ast \in \bigcap_{n \in \mathbb{N}} C_n.$ Since $x_\ast
\in f_{i_\ast}^{-1}(y),$ we deduce that $y \in \bigcup_{i \in I}
f_i(\bigcap_{n \in \mathbb{N}} C_n) = F_\s(\bigcap_{n \in \mathbb{N}}
C_n),$ which concludes the proof that $F_\s$ is $\supseteq-$continuous.\\
Thus, the conditions of the Tarski-Kantorovitch fixed-point principle are
satisfied and the first part of the theorem follows from a direct
application of this principle.\\
If $A$ is a fixed point of $F_\s,$ then $A = F_\s(A),$ so $A = F_\s^n(A)$
for all $n \in \mathbb{N}.$ Hence, $A = \bigcap_{n \in \mathbb{N}}
F_\s^n(A) \subseteq \bigcap_{n \in \mathbb{N}} F_\s^n(X),$ so that
$\bigcap_{n \in \mathbb{N}} F_\s^n(X)$ is indeed the greatest fixed point
of $F_\s.$ The last part of the theorem is trivial.
\end{proof}

\begin{remark}
Consider $X := [0, 1),$ $(f_m)_{m \geq 2} \subseteq X^X,$ $f_m(x) := \{x +
\frac{1}{m}\},$ where $\{x\}$ symbolises the fractional part of $x$ and
$(C_n)_{n \geq 3}, C_n := [0, \frac{1}{n}] \cup \br [1 - \frac{1}{n}, 1).$
Obviously, the fibres of each $f_m$ are finite, but the system is not
non-overlapping. Clearly, $\bigcap_{n \geq 3} C_n = \{0\},$ so $\bigcup_{m
\geq 2} f_m(\bigcap_{n \geq 3} C_n) = \{\frac{1}{2}, \frac{1}{3},
\dots\}.$ Also note that $0 \in f_n(C_n)$ for all $n \geq 3.$ It follows
that $0 \in \bigcup_{m \geq 2} f_m(C_n)$ for all $n \geq 3,$ so $0 \in
\bigcap_{n \geq 3} \bigcup_{m \geq 2} f_m(C_n) \setminus \bigcup_{m \geq
2} f_m(\bigcap_{n \geq 3} C_n).$ What this simple example tells us is that
we actually need to assume that the IIFS considered in the statement of
the previous theorem is non-overlapping, otherwise the Hutchinson-Barnsley
operator need not be continuous with respect to $\supseteq.$
\end{remark}

\section{The Hutchinson-Barnsley operator on \texorpdfstring{$(\f(X), \supseteq)$}{}}

In what follows $X$ will denote a Hausdorff topological space and $(\f(X),
\supseteq)$ is the poset of nonempty closed parts of $X$ ordered by
$\supseteq.$ In order to apply the Tarski-Kantorovitch fixed-point
principle, we need to have that each countable chain in $(\f(X),
\supseteq)$ admits a supremum. As shown in Proposition $4$ from \cite{8},
this assumption restricts our attention to countably compact spaces. In
fact we will be looking at countably compact sequential spaces $X$ and
establish results regarding the operators $F_\s$ and $G_\s$ in this
setting. Note that in such a space, given a decreasing sequence $(C_n)_{n
\in \mathbb{N}} \subseteq \f(X),$ its supremum is simply $\bigcap_{n \in
\mathbb{N}} C_n.$

\begin{remark}
As $X$ is a countably compact sequential space, Theorem 2.1 assures us
that a function $f : X \xrightarrow{} X$ is continuous if and only if
$f(A)$ is closed for each $A \in \f(X)$ and all fibres of $f$ are closed.
\end{remark}

The main result of this section is the following:

\begin{theorem}
Let $X$ be a countably compact sequential space and \br $\s = (X, (f_i)_{i \in
I})$ a locally finite non-overlapping IIFS, where each $f_i$ is
continuous. Then $F_\s(\f(X)) \subseteq \f(X),$ $\bigcap_{n \in
\mathbb{N}} F_\s^n(X)$ is nonempty and closed and it is the greatest fixed
point of $F_\s.$ Moreover, if $X$ is metrizable, then the sequence
$(F_\s^n(X))_{n \in \mathbb{N}} \subseteq \f(X)$ converges to $\bigcap_{n
\in \mathbb{N}} F_\s^n(X)$ with respect to the Hausdorff-Pompeiu metric.
\end{theorem}

\begin{proof}
First we prove that we can indeed consider the restriction \br $F_\s : \f(X)
\xrightarrow{} \f(X).$ Thus, let $A \in \f(X).$ We shall prove that
\br $F_\s(A) = \bigcup_{i \in I} f_i(A) = \bigcup_{i \in I} \closure{f_i(A)} =
\closure{\bigcup_{i \in I} f_i(A)} = \closure{F_\s(A)}.$ The direct
inclusion is trivial by definition, so we only prove the converse. It
suffices to show that $F_\s(A)$ is sequentially closed.\\
Indeed, let $y \in \closure{\bigcup_{i \in I} f_i(A)}$ and $(y_n)_{n \in
\mathbb{N}} \subseteq \bigcup_{i \in I} f_i(A)$ such that $y_n \to y.$ Let
$V_y$ be the neighbourhood of $y$ provided by the local-finiteness of the
IIFS, i.e. $\#\{i \in I: V_y \cap f_i(X) \neq \oldemptyset\} < \infty.$
Then $V_y$ intersects only finitely many of the sets $(f_i(A))_{i \in I}.$
Since $y_n \to y,$ we may assume that $(y_n)_{n \in \mathbb{N}} \subseteq
V_y.$ Define $i_n \in I$ to be the subscript such that $y_n \in
f_{i_n}(A)$ (it is well defined since this IIFS is non-overlapping). It
follows that $\#\{i_n: n \in \mathbb{N}\} < \infty,$ so we may find a
subsequence $(y_{n_k})_{k \in \mathbb{N}} \subseteq (y_n)_{n \in
\mathbb{N}}$ such that $i_{n_k} = i_{n_1}$ and $y_{n_k} \in
f_{i_{n_1}}(A)$  for all $k \in \mathbb{N}.$ Since $X$ is Hausdorff, we deduce that $y_{n_k} \to y,$ so that $y \in \closure{f_{i_{n_1}}(A)}
\subseteq \bigcup_{i \in I} \closure{f_i(A)} = \bigcup_{i \in I} f_i(A) =
F_\s(A),$ proving that $F_\s(A)$ is indeed closed.\\
$F_\s$ is clearly increasing and the continuity of this operator with
respect to $\supseteq$ is shown in the same way as in Theorem 3.1.\\
It is trivial that $F_\s(X) \subseteq X.$ Hence, all the conditions stated
in the Tarski-Kantorovitch fixed-point principle are satisfied and the
first part of the theorem follows directly from this.\\
For the last part of the theorem, note that if $X$ is metrizable, the
sequence $(F_\s^n(X))_{n \in \mathbb{N}}$ is $\supseteq-$increasing, so it
converges to $\sup_{n \in \mathbb{N}} F_\s^n(X) = \br \bigcap_{n \in
\mathbb{N}} F_\s^n(X)$ with respect to the Hausdorff-Pompeiu metric (see
\cite{4}).
\end{proof}

In the proof of the last theorem, we also obtained a result about the
other Hutchinson-Barnley operator, $G_\s.$

\begin{corollary}
Let $X$ be countably compact and sequential space, $\s = (X, (f_i)_{i \in
I})$ a locally finite non-overlapping IIFS. Then $F_\s(A) = \bigcup_{i \in
I} f_i(A) = \bigcup_{i \in I} \closure{f_i(A)} = \closure{\bigcup_{i \in
I} f_i(A)} = G_\s(A)$ for all $A \in \f(X).$
\end{corollary}

\begin{remark}
Let $X = [0, 1]$ (which is compact and sequential) and $(f_m)_{m \in
\mathbb{N}} \subseteq X^X$ given by $f_m(x) := \frac{x + 1}{2^{2m - 1}}$
for all $x \in X$ and $m \in \mathbb{N}.$ Obviously, the image of $f_m$ is
the closed interval $[\frac{1}{2^{2m - 1}}, \frac{1}{2^{2m - 2}}]$ and
$f_m$ is continuous for any $m \in \mathbb{N}.$ It clearly follows that
the IIFS $\s = (X, (f_m)_{m \in \mathbb{N}})$ is non-overlapping, but not
locally finite (because it is not locally finite at $0$). It is also clear
that $0 \in \closure{\bigcup_{m \in \mathbb{N}} f_m(X)} \setminus
(\bigcup_{m \in \mathbb{N}} f_m(X)),$ so in this case it is not true that
$F_\s(\f(X)) \subseteq \f(X).$ This example shows us that we cannot drop
the condition that the IIFS considered in Theorem 4.1 is locally finite,
because in that case we could have that $F_\s(\f(X)) \not \subseteq \f(X)$
and we wouldn't be able to apply the Tarski-Kantorovitch fixed-point
principle.\\
Also note that there exist locally finite non-overlapping IIFS's. Indeed,
for a nonempty set $I$ consider the IIFS $\s = (\Lambda(I), (F_i)_{i \in
I}).$ This is clearly non-overlapping. Moreover, if $\omega \in
\Lambda(I)$ is arbitrary, consider the open set $U_\omega := \{ \eta \in
\Lambda(I) : d_\Lambda(\omega, \eta) < \frac{1}{3}\}.$ If $\eta \in
\Lambda(I)$ and $\eta_1 \neq \omega_1,$ then $d_\Lambda(\omega, \eta) \geq
\frac{1}{3},$ so $\eta \not\in U_\omega.$ It follows that $U_\omega$ only
intersects $\Lambda_{\omega_1} = F_{\omega_1}(\Lambda(I)),$ so this IIFS
is also locally finite.
\end{remark}

\begin{remark}
Finally, note that we can extend Theorem 4.1 and Corollary 4.1 on the
poset $(\be(X), \supseteq)$ if $X$ is metrizable, the IIFS considered is
also bounded and there exists $B \in \be(X)$ such that $F_\s(B) \subseteq
B$. Indeed, note that in this case $F_\s(\be(X)) \subseteq \be(X)$ since
the IIFS is bounded and Theorem 4.1 clearly shows us that $F_\s(\f(X))
\subseteq \f(X),$ so we may apply the Tarski-Kantorovitch fixed-point
principle to the poset $(\be(X), \supseteq)$ and $F_\s.$ Note that in the
case that such a $B$ exists, then $\bigcap_{n \in \mathbb{N}} F_\s^n(B)$
is nonempty, closed and bounded and it is the greatest fixed point of
$F_\s$ contained in $B.$
\end{remark}

\section{The Hutchinson-Barnsley operator on \texorpdfstring{$(\ka(X), \supseteq)$}{}}

Henceforth $X$ will be a Hausdorff topological space and $(\ka(X),
\supseteq)$ will denote the poset of nonempty compact subsets of $X$
ordered by $\supseteq.$ Note that in this case every countable chain
admits a supremum and if $(C_n)_{n \in \mathbb{N}} \subseteq \ka(X)$ is a
decreasing sequence, then its supremum in this poset is simply $\bigcap_{n
\in \mathbb{N}} C_n.$

\begin{remark}
Note that we may assume that $X$ is compact. Indeed, let $\s = (X,
(f_i)_{i \in I})$ be an IIFS and $F_\s(A) := \bigcup_{i \in I} f_i(A)$ for
all $A \subseteq X$ just as before. In order to apply the
Tarski-Kantorovitch fixed-point principle to $(\ka(X), \supseteq)$ and
$F_\s,$ we would need the existence of a nonempty compact subset $B
\subseteq X$ such that $F_\s(B) \subseteq B.$ Then we may simply consider
the restricted IIFS $\s \restriction_B := (B, (f_i \restriction_B)_{i \in
I})$ and establish the desired result in the poset $(\ka(B), \supseteq).$
\end{remark}

The main result of this section is the following:

\begin{theorem}
Let $X$ be a compact space and $\s = (X, (f_i)_{i \in I})$ a locally
finite non-overlapping IIFS, where $f_i$ is continuous for all $i \in I.$
Then $F_\s(\ka(X)) \subseteq \ka(X),$ $\bigcap_{n \in \mathbb{N}}
F_\s^n(X)$ is nonempty and compact and it is the greatest fixed point of
$F_\s.$ Moreover, if $X$ is metrizable, then the sequence $(F_\s^n(X))_{n
\in \mathbb{N}} \subseteq \ka(X)$ converges to $\bigcap_{n \in \mathbb{N}}
F_\s^n(X)$ with respect to the Hausdorff-Pompeiu metric.
\end{theorem}

\begin{proof}
Note that all we need to prove is that $F_\s(\ka(X)) \subseteq \ka(X),$ as
the continuity of $F_\s$ with respect to $\supseteq$ follows in the same
way as in Theorem 3.1 and we clearly have that $F_\s(X) \subseteq X$ (and
$X$ is compact), so we may apply the Tarski-Kantorovitch fixed point
principle to establish the first part of the theorem. Also the last part
of the theorem can be proven in the same way as in Theorem 4.1.\\
To show that $F_\s(\ka(X)) \subseteq \ka(X),$ note that all compact sets
are closed in Hausdorff topological spaces, so $F_\s(\ka(X)) \subseteq
F_\s(\f(X)) \subseteq \f(X).$ But since closed subsets of compact spaces
are compact, we deduce that indeed $F_\s(\ka(X)) \subseteq \ka(X).$
\end{proof}

\begin{corollary}
Let $X$ be an arbitrary topological space and $\s = (X, (f_i)_{i \in I})$
a locally finite non-overlapping IIFS, where $f_i$ is continuous for all
$i \in I.$ The following assertions are equivalent:
\begin{enumerate}[a)]
    \item there exists $A \in \ka(X)$ such that $F_\s(A) = A;$
    \item there exists $A \in \ka(X)$ such that $F_\s(A) \subseteq A.$
\end{enumerate}
\end{corollary}

\begin{proof}
Obviously, $a) \implies b).$ The converse follows from applying Theorem
5.1 to the restricted IIFS $\s \restriction_A$ described in Remark 5.1.
\end{proof}

A direct application of Corollary 5.1 (thus a direct application of \br
Theorem 5.1) is the next result, which establishes the existence of a
fixed point of the Hutchinson-Barnsley operator associated to a locally
finite non-overlapping IIFS of contractions on a bounded Heine-Borel
metric space.

\begin{corollary}
Let $(X, d)$ be a bounded Heine-Borel metric space and \br $\s = (X, (f_i)_{i
\in I})$ a locally finite non-overlapping IIFS of contractions, where the
contractive constant of $f_i$ is $h_i \in (0, 1).$ Then there exists a
nonempty compact subset $A \subseteq X$ such that $F_\s(A) = A.$
\end{corollary}

\begin{proof}
The proof is the same as the proof of Corollary 2 from \cite{8} with the
remark that in this case, we have that $M := \sup_{i \in I} d(e_1, e_i) <
\infty,$ where $e_i$ is the fixed point of $f_i$ for each $i \in I$ and $h
:= \sup_{i \in I} h_i < 1$ by the definition of an IIFS of contractions.
\end{proof}

\section{Remarks regarding the canonical projection \texorpdfstring{$\pi : \Lambda(I)
\xrightarrow{} A(\s)$}{} for an IIFS of contractions}

Throughout this section, $(X, d)$ is a complete metric space and $\s = (X,
(f_i)_{i \in I})$ is an IIFS of contractions on $X.$ As in Definition 2.9,
we will denote $c := \sup_{i \in I} lip(f_i) < 1.$ The attractor of $\s$
will be denoted by $A = A(\s).$ By $\Lambda(I)$ we mean the shift space
associated to this IIFS (as in Definition 2.11) and $\pi : \Lambda(I)
\xrightarrow{} A(\s)$ is the canonical projection from the shift space to
the attractor of $\s.$ Note that each metric space is sequential.

\begin{proposition}
With the notations above, if $c \leq \frac{1}{3},$ then $\pi$ is a
contraction and $lip(\pi) \leq 3 \delta(A).$
\end{proposition}

\begin{proof}
Indeed, let $\alpha, \beta \in \Lambda(I),$ $\alpha \neq \beta$ and write
$\alpha = \alpha_1 \alpha_2 \dots \alpha_n \alpha_{n + 1} \dots,$ $\beta =
\beta_1 \beta_2 \dots \beta_n \beta_{n + 1} \dots.$ Define $m := \max \{i
\geq 0 : \alpha_{i} = \beta_{i}\},$ where we define $\alpha_0 = \beta_0 :=
\lambda.$ Then $\alpha_j = \beta_j$ for all $0 \leq j \leq m$ and
$\alpha_{m + 1} \neq \beta_{m + 1}.$ It follows from the definition of
$d_\Lambda$ that $\frac{1}{3^{m + 1}} \leq d_\Lambda(\alpha, \beta) \leq
\sum_{j \geq m + 1} \frac{1}{3^j} = \frac{1}{2}\frac{1}{3^m}.$ Moreover,
note that $a_\alpha, a_\beta \in \closure{A_{[\alpha]_m}} =
\closure{A_{[\beta]_m}},$ so $d(\pi(\alpha), \pi(\beta)) = d(a_\alpha,
a_\beta) \leq \delta(\closure{A_{[\alpha]_m}}) \leq c^m \delta(A)$ (by
part $a)$ of Theorem $2.3$). Thus, $d(\pi(\alpha), \pi(\beta)) \leq c^m
\delta(A) \leq \frac{1}{3^m} \delta(A) = \frac{1}{3^{m + 1}} 3 \delta(A)
\leq 3 \delta(A) d_\Lambda(\alpha, \beta).$ Since the inequality is also
valid when $\alpha = \beta,$ the conclusion follows.
\end{proof}

\begin{proposition}
If $X$ is countably compact and $\s$ is also locally finite and
non-overlapping, then the canonical projection $\pi : \Lambda(I)
\xrightarrow{} A$ is surjective.
\end{proposition}

\begin{proof}
It follows directly from Remark $4.3$, Theorem $4.1$, Corollary $4.1$, the
definition of the attractor of this IIFS and Proposition $5.1$ from
\cite{10} (stating that $\pi$ is onto if and only if $A = \bigcup_{i \in
I} f_i(A)$).
\end{proof}

\begin{remark}
Note that if the IIFS considered has the property that $f_\omega(B) \cap
f_\gamma(B) = \oldemptyset$ for all $B \in \be(X)$ and $\omega, \gamma \in
\Lambda^\ast(I),$ $\omega \neq \gamma$ (we shall say in this case that the
IIFS is strongly non-overlapping) and $X$ and $\s$ satisfy the conditions
in the last proposition, then $\pi$ is also injective. Indeed, let
$\omega, \gamma \in \Lambda(I),$ $\omega \neq \gamma$ and let $m \in
\mathbb{N}$ be such that $\omega_m \neq \gamma_m.$ Since $\pi$ is
surjective, Proposition $5.1$ from \cite{10} tells us that $f_\alpha(A) =
A_\alpha = \bigcup_{\omega \in \Lambda(I)} \{a_{\alpha \omega}\}$ for any
$\alpha \in \Lambda^\ast(I).$ Then $[\omega]_m \neq [\gamma]_m,$ so
$f_{[\omega]_m}(A) \cap f_{[\gamma]_m}(A) = A_{[\omega]_m} \cap
A_{[\gamma]_m} = \oldemptyset.$ But $a_\omega \in A_{[\omega]_m}$ and
$a_\gamma \in A_{[\gamma]_m},$ so it follows that $\pi(\omega) = a_\omega
\neq a_\gamma = \pi(\gamma),$ i.e. $\pi$ is injective.
\end{remark}

The last remark proves the following:

\begin{proposition}
If $X$ is countably compact and $\s$ is also locally finite and strongly
non-overlapping, then the canonical projection $\pi : \Lambda(I)
\xrightarrow{} A$ is bijective.
\end{proposition}

Finally, we will give sufficient conditions for the canonical projection to be
a homeomorphism and give a few corollaries.

\begin{theorem}
Let $(X, d)$ be a complete metric space and $\s = (X, (f_i)_{i \in I})$ an
IFS of bi-Lipschitz contractions with attractor $A = A_\s \in \be(X)$
admitting a seed space such that:
\begin{enumerate}[a)]
    \item the coding map $\pi: I^\omega \xrightarrow[]{} A_\s$ is continuous
    and bijective (in particular, this implies that $\s$ satisfies (SSC)- the strong separation condition);
    \item if $c_{ij} := \inf_{x, y \in A} d(f_i(x), f_j(y)) > 0$ (from (SSC)),
    we ask that $c := \inf_{i, j \in I} c_{ij} > 0.$
\end{enumerate}
Then $\pi : I^\omega \xrightarrow[]{} A$ is a homeomorphism.
\end{theorem}

\begin{proof}
All we need to show is that the inverse of $\pi$ is continuous, i.e.
$\pi^{-1}: A \xrightarrow[]{} I^\omega$ is continuous. Let $l, L \in (0, 1)$
such that $l d(x, y) \leq d(f_i(x), f_i(y)) \leq L d(x, y)$ for all $i \in
I$ and $x, y \in X.$\\
Let us fix $\varepsilon > 0$ and define $\delta_\varepsilon := c \cdot
l^{-\log_3 \varepsilon}.$ For $\alpha \in I^\omega,$ we want to show that
$\pi^{-1}(A \cap B_X(a_\alpha, \delta_\varepsilon)) \subseteq B_I(\alpha,
\varepsilon).$ We define $n(\cdot, \cdot)$ in the following way:
$n(\alpha, \beta) := \sup \{n \geq 0: [\alpha]_n = [\beta]_n\}$ for all
$\alpha, \beta \in I^\omega.$ Let us fix $\alpha, \beta \in I^\omega,$
$\alpha \neq \beta$ and $n > n(\alpha, \beta) =: p.$ Then we have that
$d(f_{[\alpha]_n}(x), f_{[\beta]_n}(y)) = d(f_{[\alpha]_p}(f_{\alpha[p,
n]}(x)), f_{[\beta]_p}(f_{\beta[p, n]}(y))) \geq l^p d(f_{\alpha_{p + 1}}(u), f_{\beta_{p + 1}}(v)),$ where $u$ and $v$ are some elements in the
attractor of $\s.$ Then by hypothesis we get \newline $d(f_{[\alpha]_n}(x),
f_{[\beta]_n}(y)) \geq l^p c_{\alpha_{p + 1} \beta_{p + 1}} \geq l^p \cdot
c > 0.$ Keeping in mind the statement of $c)$ of Theorem $2.3,$ we deduce
that if $d(a_\alpha, a_\beta) < \delta_\varepsilon,$ then we must have that
$l^p \cdot c < \delta_\varepsilon = c \cdot l^{- \log_3 \varepsilon}.$
Therefore, we infer that $p > - \log_3 \varepsilon.$ Consequently, we have
that $d_I(\alpha, \beta) < \frac{1}{3^{n(\alpha, \beta)}} = \frac{1}{3^p} <
3^{\log_3 \varepsilon} = \varepsilon.$ Hence the desired inclusion:
$\pi^{-1}(A \cap B_X(a_\alpha, \delta_\varepsilon)) \subseteq B_I(\alpha,
\varepsilon).$
\end{proof}

\begin{remark}
\begin{enumerate}[i)]
    \item Note that the $\delta_\varepsilon$ we defined in the proof of the
    previous theorem does not depend on $\alpha \in I^\omega,$ so $\pi^{-1}$
    is actually uniformly continuous;
    \item Recall that in a metric space, a set is compact if and only if it is
    complete and totally bounded;
    \item Note that if $\# I < \infty,$ then the second assumption is
    superfluous since it is always true. We want to explain why the last
    theorem is not necessarily very restrictive. One of the main points of the
    theorem is that the system considered consists in bi-Lipschitz function,
    which we have seen that is pivotal point of the proof. A large class of
    interest in the theory of iterated function systems is that of
    self-similar systems, i.e. systems of similarities. Obviously, every
    similarity is in particular a bi-Lipschitz function, so the class of
    systems considered is larger than that of self-similar systems. The second
    important assumption is that the canonical projection is continuous and
    bijective. If $I$ is finite, then the only condition here is that the
    system satisfies (SSC), which is not a big ask. If $I$ is infinite, then
    we also ask that this projection is surjective. Once again, a sufficiently
    large class of systems satisfy this condition. Finally, we asked that
    $\inf_{i, j \in I} c_{i j} > 0.$ We are not entirely sure how much this
    reduces the class of functions considered when $I$ is infinite. However,
    having gained some insight from the proof of the last theorem, we actually
    deduce a lesser condition which allows us to conclude that the inverse of
    the coding map is continuous. More exactly, we want the following
    condition to hold: given $\alpha \in I^\omega,$ the number $c_\alpha :=
    \inf_{n \in \mathbb{N}, j \in I} c_{\alpha_n j}$ is strictly positive. In
    this case we lose the uniform continuity of $\pi^{-1},$ but what matters
    is that $\pi^{-1}$ is still continuous;
    \item Note that $I^\omega$ is compact if and only if $I$ is finite.
    Indeed, since $I^\omega$ is a metric space, it is easier to prove that
    $I^\omega$ is sequentially compact if and only if $I$ is finite, which is
    fairly easy to see.
\end{enumerate}
\end{remark}

\begin{corollary}
If $\pi : I^\omega \xrightarrow[]{} A_\s$ is a homeomorphism, then the
attractor of $\s$ is totally disconnected. 
\end{corollary}

\begin{proof}
This follows immediately from the fact that $I^\omega$ is totally disconnected
in the topology induced by $d_I.$
\end{proof}

\begin{corollary}
All finite iterated function systems of bi-Lipschitz functions satisfying
(SSC) have the property that their attractor is homeomorphic to their
associated shift space.
\end{corollary}

\begin{proof}
This follows immediately from the comments made in the last remark and the
last theorem.
\end{proof}

\begin{corollary}
Let $\s$ be an IIFS of bi-Lipschitz functions satisfying (SSC) and whose
coding map is surjective. Assume that the attractor of $\s$ is compact. Then
$\pi$ cannot be a homeomorphism. In particular, neither condition $b),$ nor
the condition stated in the previous remark holds.
\end{corollary}

\begin{corollary}
If $\s$ is an IIFS of bi-Lipschitz functions satisfying the conditions of the
last theorem, then the attractor of $\s$ is not totally bounded.
\end{corollary}

\begin{proof}
Obviously, it would be true that $A$ is homeomorphic to the shift space of
$\s.$ But $I^\omega$ is not compact. Therefore, neither is the attractor of
$\s.$ However, $A$ is closed in a complete metric space, so it is also
complete. Since it is not compact, we deduce that it is not totally bounded.
\end{proof}

\end{document}